\documentclass[10pt,reqno]{amsart}

\usepackage[utf8]{inputenc}
\usepackage{amsfonts}
\usepackage{amsmath}
\usepackage{mathrsfs}
\usepackage{amsthm}
\usepackage{amssymb}
\usepackage[top=30truemm,bottom=30truemm,left=30truemm,right=30truemm]{geometry}
\usepackage[bookmarks=false,draft=false,breaklinks,colorlinks]{hyperref}
\usepackage[dvipdfmx]{graphicx}
\usepackage[all]{xy}

\usepackage{color}
\usepackage{url}

\usepackage{comment}

\usepackage{fancyhdr}
\usepackage{enumitem}
\setenumerate{label=(\arabic*),nosep}
\setitemize{nosep}
\newlist{clist}{enumerate}{1}
\setlist*[clist]{label=(\roman*),nosep}

\makeatletter
\let\@fnsymbol\@arabic
\makeatother

\theoremstyle{definition}
\newtheorem{Def}{Definition}[section]
\newtheorem{Rem}[Def]{Remark}

\theoremstyle{plain}
\newtheorem{Thm}[Def]{Theorem}
\newtheorem{Prop}[Def]{Proposition}
\newtheorem{Lem}[Def]{Lemma}
\newtheorem{Cor}[Def]{Corollary}

\newcommand{\quotient}[2]{
\mathchoice{  \text{\raise1ex\hbox{$#1$}\!\Big/\!\lower1ex\hbox{$#2$}} }
                  {  \text{\raise1pt\hbox{$#1$}\big/\lower1pt\hbox{$#2$}} }
                  {  {#1}\,/\,{#2}  }
                  {  {#1}\,/\,{#2}  }
}

\title{A note on free divergence-free vector fields}
\author{Hyuga Ito}
\author{Akihiro Miyagawa}
 
\address{
Graduate School of Mathematics, Nagoya University, Furocho, Chikusaku, Nagoya, 464-8602, Japan
}
\email{hyuga.ito.e6@math.nagoya-u.ac.jp}
\address{Department of Mathematics, Kyoto University, Kitashirakawa Oiwake-cho, Sakyo-ku, 606-8502, Japan}
\email{miyagawa.akihiro.43v@st.kyoto-u.ac.jp}

\date{\today}

\begin{document}

\maketitle

\begin{abstract}
 We exhibit an orthonormal basis of cyclic gradients and a (non-orthogonal) basis of the homogeneous free divergence-free vector field on the full Fock space and determine the dimension of Voiculescu's free divergence-free vector field of degree $k$ or less. Moreover, we also give a concrete formula for the orthogonal projection onto the space of cyclic gradients as well as the free Leray projection. 
\end{abstract}

\allowdisplaybreaks{
\section{Introduction}
In the 1980s, Voiculescu introduced free probability theory to address the free group factor isomorphism problem (see \cite{v85-0}). Within this theoretical framework, the concept of the \textit{free semi-circular system} emerges, defined as a tuple of freely independent \textit{semi-circular distributions} given by $\frac{1}{2\pi}\sqrt{4-t^2} \ 1_{[-2,2]}\ dt$ (with the Lebesgue measure $dt$). The free semi-circular system plays a role analogous to independent Gaussian distributions, as demonstrated in the free analogues of the central limit theorem, Wick's theorem, and the Stein equation. 

In the 1990s, Voiculescu also introduced free probabilistic analogues of entropy and Fisher's information measure, naming them free entropy and free Fisher's information measure, respectively (see a survey article \cite{v02a-0}). In particular, Voiculescu \cite{v98-0} introduced the so-called \textit{non-microstate} free entropy. In this approach, Voiculescu introduced a certain non-commutative differential operator, which is called the \textit{free difference quotient} and plays the role of a non-commutative counterpart of Hilbert transform, in order to define the free Fisher's information measure. 

Then, the study of the \textit{cyclic derivative} associated with the free difference quotient naturally emerged in relation to free entropy (see \cite{v00,v02}). 
In the work \cite{v00}, Voiculescu determined the range of the cyclic gradient associated with the free difference quotient and established a certain exact sequence, which Mai and Speicher \cite{ms21} and the first-named author \cite{i23} revisited in more general contexts.
In the work \cite{v02}, Voiculescu studied more geometric aspects of cyclic gradients associated with the free difference quotient. In particular, he introduced the notion of the free divergence-free vector field (originally called $\tau$-preserving non-commutative vector fields) which is a free probabilistic analogue of the divergence-free vector field, and he showed that, for a vector in the free divergence-free vector field, the associated derivative exponentiates a one-parameter automorphism of a free group factor.

This work was motivated by Voiculescu's work \cite{v19}. The paper \cite{v19} gave a free probabilistic analogue of the Euler equation (called free Euler equation) of ideal incompressible fluids based on the technologies developed in \cite{v00,v02} with replacing Euclidean space $\mathbb{R}^n$ with a free semi-circular system $s_{1},s_{2},\dots,s_{n}$ on the full Fock space, following the method of \cite{ak98}. Recently, Jekel-Li-Shlyakhtenko \cite{jls22} extended Voiculescu's framework to tracial non-commutative smooth functions, and they connect a solution of the free Euler equation with a geodesic in the free Wasserstein manifold. In our quest for examples of (non-stationary) solutions of the free Euler equation, we realized that it is difficult to analyze the free probabilistic analogue of Leray projection (called \textit{free Leray projection}), which is an ingredient of free Euler equation and the orthogonal projection from the non-commutative $L^2$-space generated by a free semi-circular system onto the free divergence-free vector field. Hence, we tried to understand the structure of the free divergence-free vector field. In \cite{v02}, Voiculescu found a basis of the free divergence-free vector field. However, we observed that this basis does not span the whole free divergence-free vector field.

The purpose of this note is to slightly modify his argument to compute the dimension of the free divergence-free vector field and deduce an exact formula for the free Leray projection. To show this, we focus on the space of cyclic gradients whose dimension can be computed by a group action of cyclic groups on words of finite length. We hope that our results will be used to find concrete solutions to the free Euler equation in future work.    

\section*{Acknowledgment}
This study started with a question from Prof.~{Voiculescu} in his intensive KTGU lectures on the free Euler equation at Kyoto University in January 2023. He also kindly hosted the visit of the second-named author to UC Berkeley in March 2023. The authors would like to thank him for his attractive lectures and for encouraging us to write this note. 
The first-named author would like to thank his supervisor, Prof. Yoshimichi Ueda for his continuous support and encouragement during his master course.
The second-named author would like to thank his supervisor, Prof.~{Benoit Collins} for his continuous support during his PhD study. 

A. Miyagawa was supported by JSPS Research Fellowships for Young Scientists, JSPS KAKENHI Grant Number JP 22J12186, JP 22KJ1817.
\section{Preliminaries}
In this section, we recall some basic notations and some facts from \cite{v02,v19}. The full Fock space $\mathcal{F}(\mathbb{C}^n)$ over $\mathbb{C}^n$ is the Hilbert space as follows.
\[
\mathcal{F}(\mathbb{C}^n)=\mathbb{C}1\oplus\bigoplus_{k\geq1}(\mathbb{C}^n)^{\otimes k},
\]
where $1$ is the vacuum vector. Throughout this note, we fix an orthonormal basis $\{e_{1},\dots,e_{n}\}$ of $\mathbb{C}^n$, and $\{f_1,\ldots,f_n\}$ denotes the standard basis of $\mathbb{C}^n$, i.e., the $i$-th component of $f_i$ is $1$ and other components are $0$. 

For any $n\in \mathbb{N}$, we set $[n]=\{1,2,\dots,n\}$. 
We denote by $[n]^*$ the free monoid with the identity $\epsilon$ and $n$-generators $1,\ldots,n$, that is, $[n]^*=\{\epsilon\}\cup \{i_{1}i_{2}\cdots i_{k}\,|\,k\in\mathbb{N},\,i_{j}\in[n],\,1\leq j\leq k\}$. For any word $w=i_{1}\cdots i_{k}\in[n]^*$, we define the length of $w$ by $k$ (the length of $\epsilon$ is defined by $0$), and $[n]^k$ denotes the subset of $[n]^*$ which consists of all elements whose lengths are $k$. For any $w\in[n]^*$ and $k\in\mathbb{N}$, $w^k$ denotes the $k$-product of $w$, that is, $w\cdots w$. In addition, for the identity $\epsilon$ and $i_{1}\cdots i_{k}\in[n]^*$, let $e_{\epsilon}$ and $e_{i_{1}\cdots i_{k}}$ denote the vacuum vector $1$ and $e_{i_{1}}\otimes e_{i_{2}}\otimes\cdots\otimes e_{i_{k}}$, respectively.

Let $l_{j}$ and $r_{j}$ denote the left and right creation operator with respect to $e_{j}$, respectively, for each $j=1,\dots,n$. Namely, for each $w \in[n]^*$, $l_je_w$ and $r_je_w$ are given by
\begin{align*}
l_j e_w &= e_{jw},\\
r_j e_w &= e_{wj}. 
\end{align*}
Set $s_{j}=l_{j}+l_{j}^*$ for each $j=1,\dots,n$. Then, $\{s_{1},\dots,s_{n}\}$ becomes a free semi-circular system with respect to the vacuum state $\tau(\cdot):=\langle\ \cdot \ 1,1\rangle$. Let $\mathbb{C}\langle s_{1},\dots,s_{n}\rangle=\mathbb{C}^s_{\langle n\rangle}$ denote the unital (algebraic) $*$-subalgebra of $B(\mathcal{F}(\mathbb{C}^n))$ generated by $\{1\}\cup\{s_{1},\dots,s_{n}\}$ and by $M$ the von Neumann subalgebra of $B(\mathcal{F}(\mathbb{C}^n))$ generated by $\mathbb{C}^s_{\langle n\rangle}$. 

We then work in the non-commutative $L^2$-space $L^2(M,\tau)$. We have 
\begin{equation*}
    l^{k_{1}}_{i_{1}}\cdots l^{k_{p}}_{i_{p}}1=e_{i_{1}^{k_{1}}i_{2}^{k_{2}}\cdots i_{p}^{k_{p}}}=U_{k_{1}}(s_{i_{1}})\cdots U_{k_{p}}(s_{i_{p}})1,
\end{equation*}
and hence obtain the following unitary isomorphism:
\begin{equation*}
    U:L^2(M,\tau)\ni U_{k_{1}}(s_{i_{1}})\cdots U_{k_{p}}(s_{i_{p}})\mapsto e_{i_{1}^{k_{1}}i_{2}^{k_{2}}\cdots i_{p}^{k_{p}}}\in \mathcal{F}(\mathbb{C}^n)
\end{equation*}
for any $i_{1},\dots,i_{p}\in[n]$ such that $i_{j}\not=i_{j+1}$ and any $k_{1},\dots,k_{p}\in\mathbb{N}\setminus\{0\}$, where the $U_{k}(t)$ are the Chebyshev polynomials (of degree $k$) of the second kind, which are orthogonal to each other with respect to the semi-circular distribution (\cite[section 1.4]{v02}).

Let $\mathbb{C}\langle l_{1},\dots,l_{n}\rangle=\mathbb{C}^l_{\langle n\rangle}$ denote the unital subalgebra of $B(\mathcal{F}(\mathbb{C}^n))$ generated by $\{1\}\cup\{l_{1},\dots,l_{n}\}$. Now, we have two cyclic gradients $\delta^s=(\delta^s_{j})_{j=1}^{n}$ and $\delta^l=(\delta^l_{j})_{j=1}^{n}$ with respect to $s_{1},\dots,s_{n}$ and with respect to $l_{1},\dots,l_{n}$ defined by
\begin{align*}
\delta^s (s_{i_1}s_{i_2}\cdots s_{i_p})&=\sum_{j=1}^p s_{i_{j+1}}\cdots s_{i_{p}}s_{i_1}s_{i_2}\cdots s_{i_{j-1}} \otimes f_{i_j} \in (\mathbb{C}^s_{\langle n\rangle})^n, \\
\delta^l (l_{i_1}l_{i_2}\cdots l_{i_p})&=\sum_{j=1}^p l_{i_{j+1}}\cdots l_{i_{p}}l_{i_1}l_{i_2}\cdots l_{i_{j-1}} \otimes f_{i_j} \in (\mathbb{C}^l_{\langle n\rangle})^n,
\end{align*}
where $i_1,\ldots,i_p \in [n]$ and we identify $(\mathbb{C}^s_{\langle n\rangle})^n \simeq \mathbb{C}^s_{\langle n\rangle} \otimes \mathbb{C}^n$ and $(\mathbb{C}^l_{\langle n\rangle})^n \simeq \mathbb{C}^l_{\langle n\rangle} \otimes \mathbb{C}^n$.
In general, $\delta^l$ is different from $\delta^s$ as an operator, but we have the following fact:
\begin{Thm}(\cite[Theorem 7.4]{v02})\label{range}
    We have $(\delta^s\mathbb{C}^s_{\langle n\rangle})[1\oplus\cdots\oplus1]
        =(\delta^l\mathbb{C}^l_{\langle n\rangle})[1\oplus\cdots\oplus1]$ in $\mathcal{F}(\mathbb{C}^n)^n$.
\end{Thm}

Following \cite{v02,v19}, we write $\mathrm{Vect}(\mathbb{C}^s_{\langle n\rangle})=(\mathbb{C}^s_{\langle n\rangle})^n$. The next object is the main target of this note.
\begin{Def}(\cite[subsection 3.5]{v02},\cite[section 2]{v19})
    The \textit{free divergence-free vector field} (with respect to semi-circular elements) is defined as follows.
    \begin{equation*}
        \mathrm{Vect}(\mathbb{C}^s_{\langle n\rangle}|\tau)
        =\left\{(p_{1},\dots,p_{n})\in\mathrm{Vect}(\mathbb{C}^s_{\langle n\rangle})\,\middle|\,\sum_{1\leq j\leq n}\tau(p_{j}\delta^s_{j}[r])=0\mbox{ for all }r\in\mathbb{C}^s_{\langle n\rangle}\right\}.
    \end{equation*}
\end{Def}
By definition, it is clear that $\mathrm{Vect}(\mathbb{C}^s_{\langle n\rangle}|\tau)=\mathrm{Vect}(\mathbb{C}^s_{\langle n\rangle})\ominus\delta^s\mathbb{C}^s_{\langle n\rangle}$. Moreover, we have the next fact:
\begin{Thm}(\cite[Theorem 7.5]{v02})
    We have 
    \begin{equation*}
        L^2(M,\tau)^n\ominus\overline{\delta^s\mathbb{C}^s_{\langle n\rangle}}
        \simeq\mathcal{F}(\mathbb{C}^n)^n\ominus\overline{\delta^s\mathbb{C}^s_{\langle n\rangle}[1\oplus\cdots\oplus1]}=\{\left((l_{j}^*-r_{j}^*)\xi\right)_{j=1}^{n}\,|\,\xi\in\mathcal{F}(\mathbb{C}^n)\}.
    \end{equation*}
    In particular, we also have
    \begin{equation*}
        \mathrm{Vect}(\mathbb{C}^s_{\langle n\rangle}|\tau)[1\oplus\cdots\oplus1]
        =\bigoplus_{k\geq0}\mathcal{X}^{(n)}_{k},
    \end{equation*}
    where $\mathcal{X}^{(n)}_{k}\subset\left[(\mathbb{C}^n)^{\otimes k}\right]^n$, $\mathcal{X}^{(n)}_{0}=\{0\}$, $\mathcal{X}^{(n)}_{k}=\{\left((l_{j}^*-r_{j}^*)\xi\right)_{j=1}^{n}\,|\,\xi\in(\mathbb{C}^n)^{\otimes k+1}\}$.
\end{Thm}

Here is a consequence of these available facts:
\begin{Lem}(\cite[Lemma 7.7]{v02})\label{dinfrlem1}
    We have $\ker\left((\theta^l)^*|_{(\mathbb{C}^n)^{\otimes k}}\right)=\ker\left((I-R)|_{(\mathbb{C}^n)^{\otimes k}}\right)$ for $k \ge 1$, where $\theta^l$ is the linear map from $\mathcal{F}(\mathbb{C}^n)^n$ to $\mathcal{F}(\mathbb{C}^n)$ such that $\theta^l[(\xi_{1},\dots,\xi_{n})]=\sum_{j=1}^{n}(l_{j}-r_{j})\xi_{j}$ and $R$ is the cyclic permutation, that is, $R(e_{i_{1}i_{2}\cdots i_{p}})=e_{i_{p}i_{1}\cdots i_{p-1}}$.
\end{Lem}

We use the same notation $R$ for the cyclic permutation on $[n]^*\setminus \{\epsilon\}$, that is, $R(i_{1}\cdots i_{k-1}i_{k})=i_{k}i_{1}\cdots i_{k-1}$ for all $i_{1}\cdots i_{k}\in [n]^*$. 

At the end of this section, we exhibit an interesting example of vectors in the free divergence-free vector field, which is inspired by the classical case when the stream function is radially symmetric.
\begin{Prop}
For any $m\in\mathbb{N}$, we have
\[\begin{pmatrix}\delta_2(s_1^2+s_2^2)^m\\-\delta_1(s_1^2+s_2^2)^m\end{pmatrix}\in \mathrm{Vect}(\mathbb{C}^s_{\langle 2\rangle}|\tau). \]
\end{Prop}
\begin{proof}
We show this proposition by induction on $m$. When $m=1$, $(\delta_2 (s_1^2+s_2^2) , -\delta_1 (s_1^2+s_2^2))= 2(s_2,-s_1)\in  \mathrm{Vect}(\mathbb{C}^s_{\langle 2\rangle}|\tau) $. Suppose that we have $(\delta_2 f , -\delta_1 f) \in  \mathrm{Vect}(\mathbb{C}^s_{\langle 2\rangle}|\tau)$ for $f=(s_1^2+s_2^2)^k$ ($1\le k \le m$).
We use the fact that free semi-circular system $(s_1,s_2)$ satisfies the analogue of the Stein equation (cf. \cite{v98-0}):
\[\tau[s_iP(s_1,s_2)]=\tau\otimes\tau[\partial_iP(s_1,s_2)]\]
for any non-commutative polynomial $P(s_1,s_2)$ where $\partial_i: \mathbb{C}^s_{\langle 2\rangle}\to \mathbb{C}^s_{\langle 2\rangle}\otimes \mathbb{C}^s_{\langle 2\rangle}$ is the free (partial) difference quotient which is a linear map defined for each monomial $P$ by
\[\partial_i P = \sum_{P=As_iB}A\otimes B.\]
From the Leibniz rule of free difference quotients (note that $\partial_i(s_1^2+s_2^2) = s_i\otimes 1 + 1\otimes s_i$ for $i=1,2$), we have for $f=(s_1^2+s_2^2)^{m+1}$
\[ (\delta_2 f , -\delta_1 f)= (m+1)[(s_1^2+s_2^2)^m v + v (s_1^2+s_2^2)^m] \]
where $ v=(s_2,-s_1)$. Therefore, for a given $r \in \mathbb{C}^s_{\langle 2\rangle}$, we want to show
\begin{equation}\label{def_divfree}\tau[s_2 (\delta_1 r)(s_1^2+s_2^2)^m ]+\tau[s_2 (s_1^2+s_2^2)^m (\delta_1 r)]-\tau[s_1 (\delta_2 r)(s_1^2+s_2^2)^m ]-\tau[s_1 (s_1^2+s_2^2)^m(\delta_2 r) ]=0.\end{equation}
By using the formula $\tau[s_i P(s_1,s_2)]= \tau \otimes \tau [\partial_i P(s_1,s_2)]$, we have from the Leibniz rule,
\begin{align*}
\tau[s_2 (\delta_1 r)(s_1^2+s_2^2)^m ]&=\tau \otimes \tau [(\partial_2(\delta_1 r)) \cdot 1\otimes (s_1^2+s_2^2)^m]\\
& \quad +\sum_{k=1}^m \tau \otimes \tau [(\delta_1 r)(s_1^2+s_2^2)^{k-1} (s_2\otimes 1 + 1\otimes s_2)(s_1^2+s_2^2)^{m-k}]\\
\tau[s_2 (s_1^2+s_2^2)^m (\delta_1 r)]&=\tau \otimes \tau [(s_1^2+s_2^2)^m\otimes 1 \cdot (\partial_2(\delta_1 r))]\\
& \quad +\sum_{k=1}^m \tau \otimes \tau [(s_1^2+s_2^2)^{m-k} (s_2\otimes 1 + 1\otimes s_2)(s_1^2+s_2^2)^{k-1} (\delta_1 r)]\\
\tau[s_1 (\delta_2 r)(s_1^2+s_2^2)^m ]&=\tau \otimes \tau [(\partial_1(\delta_2 r)) \cdot 1\otimes (s_1^2+s_2^2)^m]\\
& \quad +\sum_{k=1}^m \tau \otimes \tau [(\delta_2 r)(s_1^2+s_2^2)^{k-1} (s_1\otimes 1 + 1\otimes s_1)(s_1^2+s_2^2)^{m-k}]\\
\tau[s_1 (s_1^2+s_2^2)^m (\delta_2 r)]&=\tau \otimes \tau [(s_1^2+s_2^2)^m\otimes 1 \cdot (\partial_1(\delta_2 r))]\\
& \quad +\sum_{k=1}^m \tau \otimes \tau [(s_1^2+s_2^2)^{m-k} (s_1\otimes 1 + 1\otimes s_1)(s_1^2+s_2^2)^{k-1} (\delta_2 r)].\\
\end{align*}
Since odd moments of the free semi-circular system are zero, we have $\tau[s_i(s_1^2+s_2^2)^{m-k}] = 0$ for $i=1,2$, and thus we have
\begin{align*}
\tau[s_2 (\delta_1 r)(s_1^2+s_2^2)^m ]&=\tau \otimes \tau [(\partial_2(\delta_1 r)) \cdot 1\otimes (s_1^2+s_2^2)^m]\\
& \quad +\sum_{k=1}^m \tau  [(\delta_1 r)(s_1^2+s_2^2)^{k-1} s_2]\tau [(s_1^2+s_2^2)^{m-k}]\\
\tau[s_2 (s_1^2+s_2^2)^m (\delta_1 r)]&=\tau \otimes \tau [(s_1^2+s_2^2)^m\otimes 1 \cdot (\partial_2(\delta_1 r))]\\
& \quad +\sum_{k=1}^m \tau [(s_1^2+s_2^2)^{m-k}] \tau[s_2(s_1^2+s_2^2)^{k-1} (\delta_1 r)]\\
\tau[s_1 (\delta_2 r)(s_1^2+s_2^2)^m ]&=\tau \otimes \tau [(\partial_1(\delta_2 r)) \cdot 1\otimes (s_1^2+s_2^2)^m]\\
& \quad +\sum_{k=1}^m \tau [(\delta_2 r)(s_1^2+s_2^2)^{k-1}s_1 ]\tau[ (s_1^2+s_2^2)^{m-k}]\\
\tau[s_1 (s_1^2+s_2^2)^m (\delta_2 r)]&=\tau \otimes \tau [(s_1^2+s_2^2)^m\otimes 1 \cdot (\partial_1(\delta_2 r))]\\
& \quad +\sum_{k=1}^m  \tau [(s_1^2+s_2^2)^{m-k}]\tau[s_1(s_1^2+s_2^2)^{k-1} (\delta_2 r)].\\
\end{align*}
Then, the left-hand side of (\ref{def_divfree}) is equal to
\begin{align*}
& \tau \otimes \tau[(\partial_2(\delta_1 r)) \cdot 1\otimes (s_1^2+s_2^2)^m]+\tau \otimes \tau [(s_1^2+s_2^2)^m\otimes 1 \cdot (\partial_2(\delta_1 r))]\\
&  -\tau \otimes \tau [(\partial_1(\delta_2 r)) \cdot 1\otimes (s_1^2+s_2^2)^m]-\tau \otimes \tau [(s_1^2+s_2^2)^m\otimes 1 \cdot (\partial_1(\delta_2 r))] \\
&  + \sum_{k=1}^m\tau [(s_1^2+s_2^2)^{m-k}]\Big{(}\tau[(\delta_1 r)(s_1^2+s_2^2)^{k-1} s_2]+\tau[s_2(s_1^2+s_2^2)^{k-1} (\delta_1 r)]\\
& \quad \quad \quad  -\tau [(\delta_2 r)(s_1^2+s_2^2)^{k-1}s_1 ]-\tau[s_1(s_1^2+s_2^2)^{k-1} (\delta_2 r)]\Big{)}
\end{align*}
By using the trace property of $\tau$ and the assumption of induction $f=(s_1^2+s_2^2)^k$, the sum in the third and fourth lines is equal to $0$.
Moreover, we have from the trace property of $\tau$
\[ \tau\otimes \tau [(\partial_2(\delta_1 r)) \cdot 1\otimes X]-\tau \otimes \tau[X\otimes 1 \cdot (\partial_1(\delta_2 r))]=0\]
for any $X \in M$. By replacing $1$ and $2$ in the formula above, we can see the identity (\ref{def_divfree}) for any $r \in \mathbb{C}^s_{\langle 2\rangle}$, and thus $f=(s_1^2+s_2^2)^{m+1}$ satisfies $(\delta_2 f , -\delta_1 f) \in \mathrm{Vect}(\mathbb{C}^s_{\langle 2\rangle}|\tau)$, which completes the induction.
\end{proof}

\section{The dimension of homogeneous free divergence-free vector field}
In this section, we exhibit an orthonormal basis of $\delta^{l}(\mathbb{C}^l_{\langle n\rangle})_{k+1}\simeq\delta^{l}(\mathbb{C}^l_{\langle n\rangle})_{k+1}[1\oplus\cdots\oplus1]$ and compute the dimension of the homogeneous free divergence-free vector field $\mathcal{X}^{(n)}_{k}$ of degree $k$ for each $k\in\mathbb{N}$ and each $n\in\mathbb{N}$. 
The key point of our argument is that the cyclic gradient $\delta^l$ is invariant under the cyclic permutation, i.e., $\delta^l (l_{R u}) = \delta^l (l_{u})$ for any $u \in [n]^*$ where we write $l_u=l_{u_1}l_{u_2}\cdots l_{u_p}$ for $u=u_1u_2\cdots u_p$. Note that the cyclic permutation $R$ on $[n]^k$ induces a group action of $\mathbb{Z}_k=\quotient{\mathbb{Z}}{k\mathbb{Z}}$ on $[n]^k$. Thus, we can decompose $[n]^k$ into the orbits of this action, and we have $\delta^l (l_u) = \delta^l (l_{u'})$ if $u'$ is in the orbit $[u]=\mathbb{Z}_k u=\{gu \in [n]^k\ | \ g \in \mathbb{Z}_k\} $.   

In the following theorem, we see that $\delta^l (l_u) [1 \oplus \cdots \oplus 1]$ is orthogonal to $\delta^l (l_{u'}) [1 \oplus \cdots \oplus 1]$ if $u$ and $u'$ are not in the same orbit. Moreover, we can normalize $\delta^l (l_u) [1 \oplus \cdots \oplus 1]$ by using the order of the stabilizer subgroup $(\mathbb{Z}_k)_u = \{g \in \mathbb{Z}_k \ | \ g u = u\}$ of $u \in [n]^k$.
\begin{Thm}\label{thmbasis}
    For each $k \in \mathbb{Z}_{\ge 0}$, the subset of vectors in $\left[(\mathbb{C}^n)^{\otimes k}\right]^n $
    \[
    S_{k}=\left\{F([u]):=\frac{\delta^l(l_{u})[1\oplus\cdots\oplus1]}{|(\mathbb{Z}_{k+1})_u|\sqrt{|[u]|}}\,\middle|\,[u]\in \quotient{[n]^{k+1}}{\mathbb{Z}_{k+1}}\right\}  
    \]
    is an orthonormal basis of $\delta^{l}(\mathbb{C}^l_{\langle n\rangle})_{k+1}[1\oplus\cdots\oplus1]$, 
    where we set $l_{i_{1}i_{2}\cdots i_{p}}=l_{i_{1}}l_{i_{2}}\cdots l_{i_{p}}$.
\end{Thm}
\begin{proof}
    First, we see that $\delta^l (l_u) [1 \oplus \cdots \oplus 1]$ is orthogonal to $\delta^l (l_{u'}) [1 \oplus \cdots \oplus 1]$ if $u'\notin [u]$. Under the identification $\left[(\mathbb{C}^n)^{\otimes k+1}\right]^n \simeq (\mathbb{C}^n)^{\otimes k+1} \otimes \mathbb{C}^n$, we write the orthonormal basis of $\left[(\mathbb{C}^n)^{\otimes k+1}\right]^n $ by $\{e_w \otimes f_i\}_{w \in [n]^{k+1}, i \in [n]}$ (recall that $\{f_i\}_{i=1}^n$ denotes the standard basis). Then, for each $u=i_1i_2\cdots i_{k+1} \in [n]^{k+1}$, the cyclic derivative $\delta^l (l_u) [1 \oplus \cdots \oplus 1]$ is written by 
    \[ \sum_{j=1}^{k+1} e_{i_{j+1}\cdots i_{k+1}i_{1}\cdots i_{j-1}} \otimes f_{i_j}.\]
    If $\delta^l (l_u) [1 \oplus \cdots \oplus 1]$ is not orthogonal to $\delta^l (l_{u'}) [1 \oplus \cdots \oplus 1]$ with $u=i_1i_2\cdots i_{k+1}$ and $u'={i}_1'{i}_2'\cdots{i}_{k+1}'$, there exists $j,j' \in [k+1]$ such that $i_j={i}_{j'}'$ and 
\[ i_{j+1}\cdots i_{k+1}i_{1}\cdots i_{j-1}={i}_{j'+1}'\cdots {i}_{k+1}'{i}_{1}'\cdots {i}_{j'-1}',\]
implying that $u$ and $u'$ are in the same orbit. Therefore, $\delta^l (l_u) [1 \oplus \cdots \oplus 1]$ is orthogonal to $\delta^l (l_{u'}) [1 \oplus \cdots \oplus 1]$ if $u'\notin [u]$. 

Note that, if $p$ is the minimal number (generator) in the stabilizer subgroup $(\mathbb{Z}_{k+1})_u$ (which is also a cyclic group), then we have $u=v^{m}$ with $v=i_1i_2\cdots i_p$ and $m=|(\mathbb{Z}_{k+1})_u|$ and $p=\frac{k+1}{m}=|[u]|$. Thus, we obtain 
\[ \delta^l (l_u) [1 \oplus \cdots \oplus 1]= m \sum_{j=1}^p e_{i_{j+1}\cdots i_p v^{m-1} i_1\cdots i_{j-1}} \otimes f_{i_j}.\]
The minimality of $p$ implies that all vectors in the sum are orthonormal, and hence we have
\[ \|\delta^l (l_u) [1 \oplus \cdots \oplus 1]\|^2 = m^2 p = |(\mathbb{Z}_{k+1})_u|^2\cdot |[u]|.\]
Since $\delta^{l}(\mathbb{C}^l_{\langle n\rangle})_{k+1}[1\oplus\cdots\oplus1]$ is spanned by $\delta^l (l_u) [1 \oplus \cdots \oplus 1]$ ($u\in [n]^{k+1}$) and $F([u])$ does not depend on the choice of words in the same orbit $[u]$, we can conclude that $S_k=\{F([u]) \ | \ [u] \in \quotient{[n]^{k+1}}{\mathbb{Z}_{k+1}}\}$ is an orthonormal basis of $\delta^{l}(\mathbb{C}^l_{\langle n\rangle})_{k+1}[1\oplus\cdots\oplus1]$.
\end{proof}

\begin{Cor}\label{dimcygra}
    We have 
    $
    \dim\left(\delta^s(\mathbb{C}^s_{\langle n\rangle})_{k+1}\right)=
    \dim\left(\delta^l(\mathbb{C}^l_{\langle n\rangle})_{k+1}\right)= \left|\quotient{[n]^{k+1}}{\mathbb{Z}_{k+1}}\right|
    $, and hence
    \[
    \dim\left(\mathcal{X}^{(n)}_{k}\right)=n^{k+1}-\left|\quotient{[n]^{k+1}}{\mathbb{Z}_{k+1}}\right|
    \]
    for any $n\in\mathbb{N}$ and $k \in \mathbb{Z}_{\ge 0}$. Thus, we obtain that
    \[
    \dim
    \left(
    \mathrm{Vect}(\mathbb{C}^s_{\langle n\rangle}|\tau)_{\leq k}
    \right)
    =
    \frac{n(n^{k+1}-1)}{n-1}-\sum_{j=0}^k\left|\quotient{[n]^{j+1}}{\mathbb{Z}_{j+1}}\right|,
    \]
    where $\mathrm{Vect}(\mathbb{C}^s_{\langle n\rangle}|\tau)_{\leq k}$ is the subspace of $\mathrm{Vect}(\mathbb{C}^s_{\langle n\rangle}|\tau)$ of all elements whose degrees as polynomials with respect to $\{s_{i}\}_{i=1}^{n}$ are $k$ or less.
\end{Cor}
\begin{proof}
    It is a direct consequence from Theorem \ref{thmbasis} with the facts that $\dim\left[(\mathbb{C}^n)^{\otimes k}\right]^n=n^{k+1}$ and that $[(\mathbb{C}^n)^{\otimes k}]^n=\mathcal{X}^{(n)}_{k}\oplus\left(\delta^{l}(\mathbb{C}^l_{\langle n\rangle})_{k+1}[1\oplus\cdots\oplus1]\right)$ by Theorem \ref{range}.
\end{proof}

\begin{Rem}
The number $\left|\quotient{[n]^{k}}{\mathbb{Z}_{k}}\right|$ is equal to the number of necklaces of length $k$ such that each bead is chosen from $n$ colors. From Burnside's lemma, we have
 \[ \left|\quotient{[n]^{k}}{\mathbb{Z}_{k}}\right| = \frac{1}{k}\sum_{g \in \mathbb{Z}_k} |([n]^k)^g|, \]
where $([n]^k)^g$ is the set of elements in $[n]^k$ which are fixed by $g$. Moreover, we have \[|([n]^k)^g|=n^{\mathrm{gcd}(g,k)},\]
where $\mathrm{gcd}(g,k)$ is the greatest common divisor of $g$ and $k$.
\end{Rem}

Thanks to the orthonormal basis in Theorem \ref{thmbasis}, we can compute the orthogonal projection onto the subspace of cyclic gradients. Therefore, we can obtain a concrete formula for the free Leray projection. 
\begin{Cor}
    For any 
    $\left[
    \begin{smallmatrix}
    e_{u_{1}}\\
    \vdots\\
    e_{u_{n}}
    \end{smallmatrix}
    \right]
    \in
    (\mathcal{F}(\mathbb{C}^n))^n
    $ with $u_{j}\in[n]^{k_{j}}$, we have
    \begin{equation*}
        P_{\delta^l(\mathbb{C}^l_{\langle n\rangle})[1\oplus\cdots\oplus1]}
        \left(
        \left[
    \begin{smallmatrix}
    e_{u_{1}}\\
    \vdots\\
    e_{u_{n}}
    \end{smallmatrix}
    \right]
        \right)
        =\sum_{1\leq j\leq n}\frac{\delta^{l}(l_{ju_j})[1\oplus\cdots \oplus 1]}{k_{j}+1}.
    \end{equation*}
\end{Cor}
\begin{proof}
    It suffices to confirm the desired identity for $e_{u_{i}}\otimes f_{i} \in (\mathbb{C}^n)^{\otimes k} \otimes \mathbb{C}^n$ with $u_i \in [n]^k$ and $i \in [n]$ (recall that $\{f_i\}_{i=1}^n$ denotes the standard basis of $\mathbb{C}^n$). Since $S_{k}=\{F([w])\ | \ w \in \quotient{[n]^{k+1}}{\mathbb{Z}_{k+1}}\}$ in Theorem \ref{thmbasis} is an orthonormal basis of $\delta^{l}(\mathbb{C}^l_{\langle n\rangle})_{k+1}[1\oplus\cdots\oplus1]$, $\left\langle
    e_{u_{i}}\otimes f_{i}
    ,
    F([w])\right\rangle
    \neq 0$ implies $[w]=[iu_i]$, and $\left\langle
    e_{u_{i}}\otimes f_{i}
    ,
    F([iu_i])\right\rangle=\frac{1}{\sqrt{|[iu_i]|}} $, we have   
    \begin{align*}
        P_{\delta^l(\mathbb{C}^l_{\langle n\rangle})_{k+1}[1\oplus\cdots\oplus1]}
        \left(
        e_{u_{i}}\otimes f_{i}
        \right)
        &=
        \sum_{[w]\in  \quotient{[n]^{k+1}}{\mathbb{Z}_{k+1}}}
        \left\langle
        e_{u_{i}}\otimes f_{i}
        ,
        F([w])
        \right\rangle
        F([w])\\
        &=
        \left\langle
        e_{u_{i}}\otimes f_{i}
        ,
        F([iu_{i}])
        \right\rangle
        F([iu_{i}])\\
        &=
        \frac{1}{\sqrt{|[iu_i]|}} \cdot 
        \frac{\delta^l(l_{iu_{i}})[1\oplus\cdots\oplus1]}{|(\mathbb{Z}_{k+1})_{iu_i}|\sqrt{|[iu_i]|}}\\
        &=
        \frac{\delta^l(l_{iu_{i}})[1\oplus\cdots\oplus1]}{k+1},
    \end{align*}
    where we use the well-known identity of the group action $|[iu_i]| \cdot | (\mathbb{Z}_{k+1})_{iu_i}|=|\mathbb{Z}_{k+1}|=k+1$.
\end{proof}

\begin{Cor}
    For any 
    $\left[
    \begin{smallmatrix}
    e_{u_{1}}\\
    \vdots\\
    e_{u_{n}}
    \end{smallmatrix}
    \right]
    \in
    (\mathcal{F}(\mathbb{C}^n))^n$ with $u_{j}\in[n]^{k_{j}}$,
    we have
    \[
    \Pi
    \left(
    \left[
    \begin{smallmatrix}
    e_{u_{1}}\\
    \vdots\\
    e_{u_{n}}
    \end{smallmatrix}
    \right]
    \right)
    =
    \sum_{1\leq j\leq n}
    \left(
    e_{u_{j}}\otimes f_{j}-\frac{\delta^{l}(l_{ju_{j}})[1\oplus \cdots \oplus 1]}{k_{j}+1}
    \right),
    \]
    where $\Pi$ is the free Leray projection (on the full Fock space side) (see \cite[section 3]{v19}).
\end{Cor}

\begin{Rem}
We can also describe a (non-orthogonal) basis of the homogeneous free divergence-free vector field $\mathcal{X}^{(n)}_{k}$ on the full Fock space. Indeed, Lemma \ref{dinfrlem1} tells us
\[
(\theta^l)^*
:
\mathrm{ran}\left((I-R)|_{(\mathbb{C}^n)^{\otimes k+1}}\right)
\to
\mathrm{ran}\left((\theta^l)^*|_{(\mathbb{C}^n)^{\otimes k+1}}\right)
\]
is a linear isomorphism. Here, note that $\mathrm{ran}\left((\theta^l)^*|_{(\mathbb{C}^n)^{\otimes k+1}}\right)=\mathcal{X}^{(n)}_{k}$, and hence
\[
\dim
\left(
\mathrm{ran}\left((I-R)|_{(\mathbb{C}^n)^{\otimes k+1}}\right)
\right)
=
\dim\left(\mathcal{X}^{(n)}_{k}\right)
=
n^{k+1}-
\left|
\quotient{[n]^{k+1}}{\mathbb{Z}_{k+1}}
\right|.
\]
Thus, in order to obtain a basis of $\mathcal{X}^{(n)}_{k}$, it suffices to find a basis of $\mathrm{ran}\left((I-R)|_{(\mathbb{C}^n)^{\otimes k+1}}\right)$. In \cite{v02}, Voiculescu introduced a basis of $\mathcal{X}^{(n)}_{k}$ by $\{(\theta^l)^*(I-R)e_w \ | \ w \in \Omega_{k+1}\}$ where $\Omega_{k+1}=\{w\in [n]^{k+1} \ | \ w \prec R w, w \neq Rw\}$ and $\prec$ is the lexicographic order. However, since we have the decomposition $[n]^{k+1}=\{Rw=w\} \sqcup \Omega_{k+1} \sqcup \{w\succ Rw, w\neq Rw\}$ with $|\{Rw=w\}|=n$ and $|\Omega_{k+1}| = |\{w\succ Rw, w\neq Rw\}|$, the cardinality of $\Omega_{k+1}$ is $ \frac{1}{2}(n^{k+1}-n)$, which is smaller than our dimension $n^{k+1}-\left|\quotient{[n]^{k+1}}{\mathbb{Z}_{k+1}}\right|$. In fact, we can modify the set $\Omega_{k+1}$ and take a basis of $\mathrm{ran}\left((I-R)|_{(\mathbb{C}^n)^{\otimes k+1}}\right)$ by considering the action of $\mathbb{Z}_{k+1}$. 
\end{Rem}
\begin{Prop}\label{lembasisdivfree}
    The set $B=\bigsqcup_{[u]\in\quotient{[n]^{k+1}}{\mathbb{Z}_{k+1}}} B_{[u]}$, where
    \[
    B_{[u]}
    =
    \left\{
    (I-R)e_{v}=e_{v}-e_{Rv}\,\middle|\,v\in[u]\setminus\{u\}
    \right\},
    \]
    is a basis of $\mathrm{ran}\left((I-R)|_{(\mathbb{C}^n)^{\otimes k+1}}\right)$, and thus, 
     the set $\widetilde{B}=\bigsqcup_{[u]\in\quotient{[n]^{k+1}}{\mathbb{Z}_{k+1}}}\widetilde{B}_{[u]}$, where
    \[
    \widetilde{B}_{[u]}
    =
    \left\{
    \left(\delta_{j,i_{1}}e_{i_{2}\cdots i_{k+1}}-2\delta_{j,i_{k+1}}e_{i_{1}\cdots i_{k}}+ \delta_{j,i_k} e_{i_{k+1}i_1\cdots i_{k-1}}\right)_{1\leq j\leq n}
    \,\middle|\,
    i_{1}i_{2}\cdots i_{k+1}\in[u]\setminus\{u\}
    \right\},
    \]
    is a basis of $\mathcal{X}^{(n)}_{k}$.
\end{Prop}
\begin{proof}
    Since 
    $\left|B\right|
    =
    n^{k+1}-\left|\quotient{[n]^{k+1}}{\mathbb{Z}_{k+1}}\right|
    =
    \dim
    \left(
    \mathrm{ran}\left((I-R)|_{(\mathbb{C}^n)^{\otimes k+1}}\right)
    \right)$, we have to confirm only that all elements of $B$ are linearly independent. Remark that if $[u]\not=[v]$ in $\quotient{[n]^{k+1}}{\mathbb{Z}_{k+1}}$, then $B_{[u]}$ and $B_{[v]}$ are orthogonal to each other. Hence, it suffices to show that all the elements of $B_{[u]}$ are linearly independent of each other for each $[u]\in\quotient{[n]^{k+1}}{\mathbb{Z}_{k+1}}$. 
    
    Choose an arbitrary $[u]\in\quotient{[n]^{k+1}}{\mathbb{Z}_{k+1}}$. Assume that $\sum_{v\in[u]\setminus\{u\}}\alpha(v)\cdot (I-R)e_{v}=0$ in $(\mathbb{C}^n)^{\otimes k+1}$ with $\alpha(v)\in\mathbb{C}$. Let us write 
    \[
    [u]\setminus\{u\}
    =
    \{
    v,Rv,R^2v,\dots,R^pv
    \}\subset(\mathbb{C}^n)^{\otimes k+1}\quad(p=[u]-2\mbox{ and }R^{p+1}v=u).
    \]
    Remark that $R^iv\not= R^jv$ for any $0\leq i\not=j\leq p$.
    Then, we observe that
    \begin{align*}
    \alpha(v)e_{v}
    +
    \sum_{1\leq j\leq p}(\alpha(R^jv)-\alpha(R^{j-1}v))\cdot e_{R^jv}
    -
    \alpha(R^pv)u=0.
    \end{align*}
    By the linear independence of $\{e_{u}\}_{u\in[n]^*}$, we have $\alpha(R^j)=0$ for all $0\leq j\leq p$. Applying $(\theta^l)^*=(l_j^*-r_j^*)_{1\le j\le n}$ to the basis $B$, we obtain $\tilde{B}$ as a basis of $\mathcal{X}^{(n)}_{k}$. 
\end{proof}
\begin{Rem}
In \cite{v02}, Voiculescu introduced another basis of $\mathcal{X}^{(n)}_{k}$.
For $I=i_0\cdots i_k \in [n]^{k+1}$, let $\mathrm{per}(I)$ be the period of $I$, i.e. the least $m \in \{1,\ldots,k+1\}$ such that $i_s=i_t$ whenever $s\equiv t  \ (\mathrm{mod} \ m)$. Let $\rho(\mathrm{per}(I))$ be the set of the non-unital roots of $\zeta^{\mathrm{per}(I)}= 1$ (we set $\rho(1)=\emptyset$).
Then, Voiculescu introduced the following set for the basis of $\mathcal{X}^{(n)}_{k}$
\[ \left\{\sum_{k=1}^{\mathrm{per}(I)-1}\zeta^j F_{i_ji_{j+1}\cdots i_k i_0 \cdots i_{j-1}}\ | \ I=i_0\cdots i_k \in \omega(k+1), \zeta \in \rho(\mathrm{per}(I))  \right\}, \]
where $F_w=(\theta^l)^*(I-R)e_w$ for $w \in [n]^*$ and $\omega(k+1)$ is defined by
\[\left\{I=i_0\cdots i_k \in [n]^{k+1}\ | \ i_0\cdots i_k \prec i_ji_{j+1}\cdots i_k i_0 \cdots i_{j-1}, \ j=1,\ldots,k \right\}.\]
Note that the element $I$ in $\omega(k+1)$ is the minimal element in the orbit of $I$ with respect to the lexicographic order, and thus there is a bijection between $\omega(k+1)$ and $\quotient{[n]^{k+1}}{\mathbb{Z}_{k+1}}$. Since $|\rho(\mathrm{per}(I))|=\mathrm{per}(I)-1=|[I]|-1$, the number of $(I,\zeta)$ such that $I\in \omega(k+1)$ and $ \zeta \in \rho(\mathrm{per}(I))$ is equal to $n^{k+1}-\left|\quotient{[n]^{k+1}}{\mathbb{Z}_{k+1}} \right|$, which coincides with our dimension. 
\end{Rem}
}

\end{document}